\documentclass{amsart}

\usepackage{amsmath}
\usepackage{amssymb}
\usepackage{mathrsfs}
\usepackage{amsthm}

\usepackage{tikz}

\title{Extensions of Continuous Function in $LG$-Topology}

\author[M. Badie]{Mehdi Badie}
\address{Department of Basic Sciences, Jundi-Shapur University of Technology, Dezful, Iran}
\email{badie@jsu.ac.ir}

\author[H. Kasiri]{Hossein Kasiri}
\address{Department of Basic Sciences, Jundi-Shapur University of Technology, Dezful, Iran}
\email{hossein\_kasiry@jsu.ac.ir}

\theoremstyle{plain}
\newtheorem{Thm}{Theorem}[section]
\newtheorem{Lem}[Thm]{Lemma}
\newtheorem{Pro}[Thm]{Proposition}
\newtheorem{Cor}[Thm]{Corollary}
\theoremstyle{definition}
\newtheorem{Exa}[Thm]{Example}

\newtheorem{Def}[Thm]{Definition}

\newcommand{\ff}{if and only if }

\begin{document}
	
	\begin{abstract}
		In this article we introduce three maps, $ OLG $, $ CLG $ and $ LG $ in $ LGT $-space literature and show that these maps are extension of the continuous function in $ LGT $-spaces and have the almost properties of the continuous functions. Also, it has been introduced and studied the natural notions, quotient, decomposition, weak $LG$-topology and isomorphism, related to the continuous function.
	\end{abstract}
	
	\maketitle
	
	\section{Introduction}
	
	Two kind of pointfree version of topological spaces have been introduced. In the first model, since the set of open sets of a topological space is a frame, the researchers focus on a topology as a frame. Introducing of the first version has been started in \cite{wallman1938lattices,mckinsey1944algebra,nobeling1954,ehresmann1957gattungen,papert1964abstract,dowker1966quotient,isbell1973atomless,simmons1978framework} and then studied in many articles. In the new model, researchers pursue this viewpoint and introduce the new structure. The second structure, introduced and studied just in \cite{aliabad2015lg} and called $ LG $-topology. In this article, we continue this studying. We introduce and study, three extensions of continuous maps in $ LG $-topology and some notions related to the continuity. Also, we show that they have the almost properties of the well-known notions in the topology literature.
	
	In the rest of this section we recall some pertinent definitions and give some elementary properties of the right adjoint of a map which we need them in the main parts of this research. In Section 2, we introduce three kind maps $ OLG $, $ CLG $ and $ LG $ maps. We show that these maps have some properties similar to the continuous function and also we prove that they are extensions of the continuous function in the lattice generalization topology literature and finally, we introduce a functor form topological spaces and continuous functions to $ LGT $-space and $ LG $-maps. Section 3 is devoted to properties of $ OLG $, $ CLG $ and $ LG $ maps on subspaces of an $ LG $-spaces. In Section 4, the relation of the $ OLG $ and product $ LGT $-space studied and then the new notions quotient $ LG $-space and decomposition $ LG $-space are introduced, by the used of $ OLG $, $ CLG $ and $ LG $ maps. Finally, in Section 5, the concepts open map, $ LG $ isomorphism and related topics introduced and studied.
	
	A lattice is said to be \emph{complete} if every subset of $ L $ has the supremum, so a complete lattice has the greatest element $ 1 $  and the smallest element $ 0 $. A \emph{frame} $ F $ is a complete lattice such that for each $ a \in F $ and $ \{ b_i \}_{i \in I} \subseteq F  $,  $ a \wedge \big(  \bigvee_{ i \in I } b_i \big) = \bigvee_{i \in I} (‌a \wedge b_i) $, if $ a \vee \big(  \bigwedge_{ i \in I } b_i \big) = \bigwedge_{i \in I} (‌a \vee b_i) $, also holds, then $ F $ is called \emph{symmetric frame}. For each element $ a \in F $,  $ a^* $ is defined $ \bigvee a^\perp $, where $ a^\perp = \{ x \in F : x \wedge a = 0 \} $. An element $ b $ of a frame $ F $ is called the \emph{complement} of an element $ a  $ of $ F $, if $ a \wedge b = 0 $ and $ a \vee b = 1 $. Clearly, the complement of an element $ a $, if exists, is unique. We denote the complement of $ a $ by $ a^c $. It is easy to see that if the complement of an element $ a $ exists, then $ a^* = a^c $. If each element of a frame $ F $ has a complement, we say that $ F $ is a \emph{complemented} frame. A subset $ G $ of a frame $ F $ is said to be a \emph{subframe} of $ F $, if $ G $ is closed under finite meets and arbitrary joins. 
	
	Suppose that $ F $ is a frame, we say $ \tau \subseteq F $ is a \emph{lattice generalized topology} (briefly \emph{$LG$-topology}) on $ F $, if $ \tau $ is a subframe of $ F $, then $ (F,\tau) $ (briefly $ F $) is said to be an \emph{$ LGT $-space}, every element of $ \tau $ is called \emph{open} element and every element of $ \tau^* = \{ t^* : t \in \tau \} $ is called \emph{closed} element. Clearly, for each family $ \mathcal{F} \subseteq \tau^*  $ we have $ \bigwedge \mathcal{F} \in \tau^* $. Furthermore, if $ \tau^* $ is sublattice of $ F $, then we call $ \tau $ is an \emph{$ LT $-space}. An $LG$-topology $ \tau $ on a frame $ F $ is said to be \emph{discrete} (\emph{trivial}), if $ \tau = F $ ($ \tau = \{ 0 , 1 \} $). If $ (F,\tau) $ is a $ LGT $-space, then for each element $ a \in F $, the \emph{interior} and \emph{closure} of $ a $, denoted by $ a^\circ $ and $ \bar{a} $, are defined by $ \bigvee \{ t \in \tau :  t \leqslant a \} $ and $ \bigwedge \{ f \in \tau^* :‌ a \leqslant f \} $, respectively. It is clear to see that, $ \overline{a} \in \tau^* $, for each $ a \in F $ and $ f^* = f $, for each $ f \in \tau^* $. If for some $ a \in F $, $ \overline{a} = 1 $, then we say $ a $ is \emph{dense} in $ F $. Clearly, if $(F, \tau)$ is an $ LGT $-space and $a \in F$, then $(F_a, \tau_a) $ is also an $ LGT $-space, in which $ F_a = \downarrow a = \{ x \in F : x \leqslant a \} $ and $ \tau_a = \{ s \wedge a: s \in \tau \}$. We call $(F_a, \tau_a) $ a \emph{subspace} of $(F, \tau)$. A $ LGT $-space $ (F,\tau) $ is called \emph{compact} if for each family of open elements $ \{‌ t_\alpha \}_{\alpha \in A} $ that $ 1 = \bigvee_{ \alpha \in A } a_\alpha $, there is a subfamily $ \{ t_{\alpha_i} \}_{i=1}^n $ of $ \{ t_\alpha \}_{\alpha \in A} $ such that $ 1 = \bigvee_{i=1}^n a_{\alpha_i} $. Similarly, \emph{countably compact} and \emph{Lindel\"of} spaces are defined. Suppose that $ \{ (‌F_\alpha,\tau_\alpha) \}_{\alpha \in A} $ is a family of $ LGT $-spaces,  then $ \tau_{_P} = \{ t \in \prod_{\alpha \in A} F_\alpha :‌  \forall \alpha \in A, \quad t_\alpha \in \tau_\alpha  \text{‌ and just for finitely many }  \alpha \in A, \; t_\alpha \neq 1  \} $  is an $LG$-topology on $ \prod_{\alpha \in A} F_\alpha $. $ (‌ \prod_{\alpha \in A} F_\alpha, \tau_{_P} ) $ is called \emph{product $ LGT $-space}. Suppose that $ (F,\tau) $ is an $ LGT $-space and $ B \subseteq \tau $, we say $ B $ is a \emph{base} for $ \tau $, if for each $ t \in \tau $, there is some $ B' \subseteq B $ such that $ t = \bigvee B' $.
	
	Suppose that $ F_1 $ and $ F_2 $ are two frames, we say a map $ \phi : F_1 \rightarrow F_2 $ is an \emph{arbitrary join preserve} map if for each $ E \subseteq F_1 $, we have $ \phi(\bigvee E) = \bigvee \phi(E) $, also we say a map $ \psi :‌F_2 \rightarrow F_1 $ is a \emph{right adjoint} of $ \phi $, if for each $ a \in  F_1 $  and $ b \in F_2  $
	\[  \phi(a) \leqslant b \qquad \Leftrightarrow \qquad  a \leqslant \psi(b)  \]
	Similarly \emph{arbitrary meet preserve} and \emph{left adjoint} are defined. It is easy to show that, if $ \phi $ is an arbitrary join (meet) preserve map, then $ \phi $ has a unique right (left) adjoint map and this right adjoint map, denoted by $ \phi_* $($ \phi^* $), is an arbitrary meet (join) preserve.
	
	The reader is referred to \cite{stephen1970general}, \cite{picado2011frames}, for undefined terms and notations.
	
	\begin{Pro}
		Suppose that $ F_1 $, $ F_2 $ and $ F_3 $ are frames, $ \phi :‌ F_1 \rightarrow F_2 $ and $ \psi :‌F_2 \rightarrow F_3 $ are two arbitrary join preserve maps. Then 
		\begin{itemize}
			\item[(a)] $ \phi_*(b) = \bigvee_{\phi(x) \leqslant b} x $, for every $ b \in F_2 $.
			\item[(b)] $ \phi_*(b) = \max\{ x \in F_1 : \phi(x) \leqslant b \}  $, for each $ b \in F_2 $. 
			\item[(c)] $ \phi (\phi_*(b))‌ = b $, for every $ b \in F_2 $.
			\item[(d)] $ \phi \circ \psi  $ is an arbitrary join preserve map and $ (‌ \phi \circ \psi)_* = \psi_* \circ \phi_* $.
		\end{itemize}
		\label{right adjoint}
	\end{Pro}
	\begin{proof}
		(a). Since $ \phi $ is an  arbitrary join preserve map, $ \phi $ has a unique right adjoint map, so it is sufficient to show that $ \upsilon :‌F_2 \rightarrow F_1 $, defined by $ \upsilon(b) = \bigvee_{\phi(a) \leqslant b} a $, is a right adjoint of $ \phi $. Now, suppose that $ \phi(a) \leqslant b $, then $ a \leqslant \bigvee_{\phi(x) \leqslant b} x = \upsilon(b) $. Conversely, If $ a \leqslant \upsilon(b) = \bigvee_{\phi(x) \leqslant b} x  $, then 
		\[  \phi(a) = \phi\big( \bigvee_{\phi(x) \leqslant b} x \big) = \bigvee_{\phi(x) \leqslant b} \phi(x) \leqslant b. \]
		Hence $ \upsilon  $ is a right adjoint of $ \phi $.
		
		(b), (c) and (d). They are evident.
	\end{proof}
	
	We can give and prove a proposition similar to the above proposition for the arbitrary meet preserve and left adjoint maps
	
	\begin{Lem}
		Suppose that $ F_1 $ and $ F_2 $ are two frames and $ \phi : F_1 \rightarrow F_2 $. If $ \phi $ and $ \phi_* $ are arbitrary join preserve maps, then $ \phi_{**} = \phi $. 
		\label{preserving and right adjoin}
	\end{Lem}
	\begin{proof}
		By Proposition \ref{right adjoint}, $ \phi_*(\phi_{**}(x)) = x $, so $ \phi(\phi_*(\phi_{**}(x))) = \phi(x) $. In the other hand, again, by Proposition \ref{right adjoint}, $ \phi(\phi_*(\phi_{**}(x))) = \phi_{**}(x) $, hence $ \phi(x) = \phi_{**}(x) $.
	\end{proof}
	
	\begin{Lem}
		Suppose that $ F_1 $ and $ F_2 $ are two frames. If $ \phi $ is a one-to-one map from $ F_1 $ onto $ F_2 $, then the following statements are equivalent
		\begin{itemize}
			\item[(a)] $ \phi $ is arbitrary join preserve.
			\item[(b)] $ \phi $ is arbitrary meet preserve.
			\item[(c)] $ \phi_* $ is arbitrary join preserve.
			\item[(d)] $ \phi_* $ is arbitrary meet preserve.
		\end{itemize}
		\label{one-to-one and preserving}
	\end{Lem}
	\begin{proof}
		It is easy, by this fact that, if a one-to-one onto map $ \phi $ has either a right or a left adjoint, then $ \phi^{-1} $ is the right and left adjoint.
	\end{proof}
	
	\section{Extensions}
	
	In this section we introduce some maps and, by some examples, show that the categories of these maps are not equal together. Then we prove that they have the most property of the continuous functions. After that, it has been shown that these introduced maps are extensions of the continuous function. Finally, a functor has  been introduced between categories of the topological spaces and the $ LGT $-spaces.
	
	\begin{Def}
		Suppose that $(F_1, \tau_{_1})$ and $(F_2,\tau_{_2})$ are two $LGT$-spaces. An arbitrary join preserve map $\phi : F_1 \rightarrow F_2 $ is called an $ OLG $ ($ CLG $) map if for each $b \in \tau_{_2}$ ($b \in \tau_{_2}^*$), $\phi_*(b) \in \tau_{_1} $ ($\phi_*(b) \in \tau_{_1}^* $). We say  $\phi$ is an $ LG $ map if it is both $ OLG $ and  $ CLG $. 
	\end{Def}
	
	Suppose that $(F_1, \tau_{_1})$ and $(F_2,\tau_{_2})$ are two $LGT$-spaces. It easy to see that, if a map $ \phi  :‌F‌_1 \rightarrow F_2 $ is $ OLG $, then for each $ u \in \tau_{_2} $,
	\[ \phi_*(u) =  \bigvee_{\phi(x) \leqslant u} x = \bigvee_{\substack{\phi(t) \leqslant u \\ t \in \tau_{_1}}} t  = \max \{ t \in \tau_{_1} : \phi(t) \leqslant u \} .\]
	
	\begin{Pro}
		Suppose that $(F, \tau_{_1})$ and $(F,\tau_{_2})$ are two $LGT$-spaces. 
		\begin{itemize}
			\item[(a)] The identity map $ I_{_F} : (F, \tau_{_1}) \rightarrow (F,\tau_{_2}) $ is $ OLG $, \ff $ \tau_{_2} \subseteq \tau_{_1} $.
			\item[(b)] The identity map $ I_{_F} : (F, \tau_{_1}) \rightarrow (F,\tau_{_2}) $ is $ CLG $, \ff $ \tau^*_{_2} \subseteq \tau^*_{_1} $.
		\end{itemize}
		\label{topology order}
	\end{Pro}
	\begin{proof}
		They are same and straightforward.
	\end{proof}
	
	In the following example, by the use of the above proposition, we show that a $ CLG $ map need not be an $ OLG $ map.
	
	\begin{Exa}
		Suppose that  $ F‌ = \{ 0 , a , 1 \} $,  $ \tau_{_1} = \{‌ 0 , 1 \} $ is the trivial $LG$-topology and $ \tau_{_2} = \{‌ 0, a , 1 \} $ is the discrete $LG$-topology. Then $ \tau_{_1}^* = \{ 0,1 \} $ and $ \tau_{_2}^* = \{‌ 0 , 1\} $. By the above proposition, $ I_F : (F,\tau_{_1})‌ \rightarrow (F,\tau_{_2}) $ is $ CLG $ and is not $ OLG $.
		\label{clg is not olg}
	\end{Exa}
	
	\begin{Cor}
		Let $ (F,\tau) $ be an $ LGT $-space.
		\begin{itemize}
			\item[(a)] $ \tau $ is the discrete $LG$-topology on $ F $; \ff for each $ LGT $-space $ (F',\tau') $, each arbitrary join preserve map $ \phi : F \rightarrow F' $ is $ OLG $.
			\item[(b)] $ \tau^* = F $; \ff for each $ LGT $-space $ (F',\tau') $, each arbitrary join preserve map $ \phi : F \rightarrow F' $ is $ CLG $.
			\item[(c)] If for each $ LGT $-space $ (F',\tau') $, each arbitrary join preserve map $ \phi : F' \rightarrow F $ is $ OLG $, then $ \tau $ is the trivial $LG$-topology on $ F $.
			\item[(d)] If for each $ LGT $-space $ (F',\tau') $, each arbitrary join preserve map $ \phi : F' \rightarrow F $ is $ CLG $, then $ \tau^* = \{ 0 , 1 \}$.
		\end{itemize}
		\label{discrete and continuous}
	\end{Cor}
	\begin{proof}
		(a $ \Rightarrow $). It is clear.
		
		(a $ \Leftarrow $). Since the identity map $ I_F : (F,\tau) \rightarrow (F,F) $ is $ OLG $, by Proposition \ref{topology order}, $ \tau = F $ is the discrete $LG$-topology on $ F $.
		
		(b). It is similar to (a).
		
		(c). Suppose that $ \tau' $ is the trivial topology on $ F $. Since $ I_F : (F,\tau') \rightarrow (F,\tau) $ is $ OLG $, by Proposition \ref{topology order}, $ \tau = \tau' $.
		
		(d). It is similar to (b).
	\end{proof}
	
	In the following example, we show that the converse of the statements (c) and (d)‌ in the above theorem need not be true.
	
	\begin{Exa}
		Suppose that
		
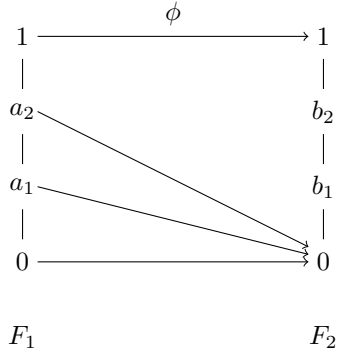
\begin{center}
			\begin{tikzpicture}
			\draw(-2,-1)node{$F_1$};
			\draw(-2,0)node{$0$};
			\draw(-2,0.3)--(-2,0.7);
			\draw(-2,1)node{$a_1$};
			\draw(-2,1.3)--(-2,1.7);
			\draw(-2,2)node{$a_2$};
			\draw(-2,2.3)--(-2,2.7);
			\draw(-2,3)node{$1$};
			
			\draw(2,-1)node{$F_2$};
			\draw(2,0)node{$0$};
			\draw(2,0.3)--(2,0.7);
			\draw(2,1)node{$b_1$};
			\draw(2,1.3)--(2,1.7);
			\draw(2,2)node{$b_2$};
			\draw(2,2.3)--(2,2.7);
			\draw(2,3)node{$1$};
			
			\draw[->](-1.8,2)--(1.8,0.2);	
			\draw[->](-1.8,1)--(1.8,0.1);	
			\draw[->](-1.8,0)--(1.8,0);	
			\draw[->](-1.8,3)--(1.8,3);	
			
			\draw(0,3.3)node{$\phi$};
			\end{tikzpicture}
		\end{center}
		and $ \tau_{_1} = \{0,a_1,1\} $ and $ \tau_{_2} $ is the trivial $LG$-topology. Then $ \tau_{_1}^* = \{‌ 0 ,1 \} $ and $ \tau_{_2}^*=\{0,1\} $. Clearly, $ \phi $ is arbitrary join preserve and it is easy to check that $ \phi  $ is not neither $ OLG $ nor $ CLG $.
	\end{Exa}
	
	\begin{Lem}
		Suppose that $ (F,\tau) $ is an $ LGT $-space. If $ \tau^* = F $, then $ \tau $ is the discrete $LG$-topology.
	\end{Lem}
	\begin{proof}
		Suppose that $ a \in F $, then $ t,u,v \in \tau $ exist such that $ a = t^* $, $ t = u^* $ and $ u = v^* $, then, by \cite[Remark 1.1]{aliabad2015lg}, 
		\[ a =  t^* = u^{**} = v^{***} = v^* = u \in \tau \]
		Thus $ \tau = F $ is the discrete $LG$-topology.   
	\end{proof}
	
	Now the above corollary and the above lemma conclude the following corollary.
	
	\begin{Cor}
		Suppose that $ (F,\tau) $ is an $ LGT $-space. If for each $ LGT $-space $ (F',\tau') $, each arbitrary join preserve is $ CLG $, then $ \tau $ is the discrete $LG$-topology.
	\end{Cor}
	
	In the following example we give an $ OLG $ map which is not $ CLG $ map and we show that the converse of the above corollary need not be true.  
	
	\begin{Exa}
		Suppose that
		
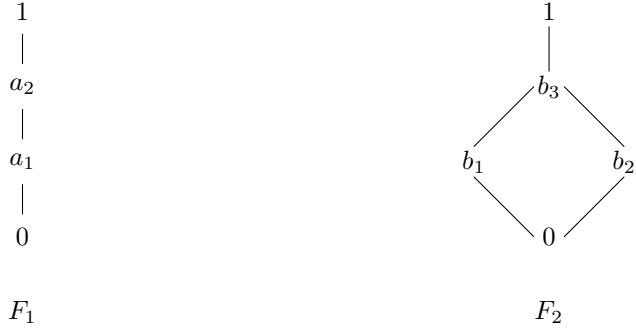
\begin{center}
			\begin{tikzpicture}
			\draw(2,-1)node{$F_1$};
			\draw(2,0)node{$0$};
			\draw(2,0.3)--(2,0.7);
			\draw(2,1)node{$a_1$};
			\draw(2,1.3)--(2,1.7);
			\draw(2,2)node{$a_2$};
			\draw(2,2.3)--(2,2.7);
			\draw(2,3)node{$1$};
			
			\draw(9,-1)node{$F_2$};
			\draw(9,0)node{$0$};
			\draw(8.8,0)--(8,0.8);
			\draw(9.2,0)--(10,0.8);
			\draw(10,1)node{$b_2$};
			\draw(8,1)node{$b_1$};
			\draw(8.8,2)--(8,1.2);
			\draw(9.2,2)--(10,1.2);
			\draw(9,2)node{$b_3$};
			\draw(9,2.2)--(9,2.8);
			\draw(9,3)node{$1$};
			\end{tikzpicture}
		\end{center} 
		and $ \tau_{_1} = F_1 $ and $ \tau_{_2} = F_2 $ are the trivial $ LG $-topologies. Then the map $ I_F : F_1 \rightarrow F_2 $, defined by $ \phi(0) = 0 $, $ \phi(a_1) = b_1 $, $ \phi(a_2) = b_3 $ and $ \phi(1) = 1 $ is an $ OLG $ map, by Corollary \ref{discrete and continuous}. But $ b_2 \in \{ 0 , b_1 , b_2 , 1 \} =  \tau_{_2}^* $ and $ \bigvee_{\phi(x) \leqslant b_2} x = a_1 \notin \{ 0 , 1 \} = \tau_{_1}^* $, so $ \phi $ is not $ CLG $ map.
	\end{Exa}
	
	\begin{Pro}
		Suppose that $(F_1, \tau_{_1})$ and $(F_2,\tau_{_2})$ are two $LGT$-spaces.  If $\phi : F_1 \rightarrow F_2$ is an $ OLG $ map and $ \phi_*(a^*) = \big(\phi_*(a)\big)^* $, then $\phi$ is an $ LG $ map.
		\label{OLG is LG}
	\end{Pro}
	\begin{proof}
		It is clear.
	\end{proof}
	
	\begin{Thm}
		Suppose that $(F_1, \tau_{_1})$ and $(F_2,\tau_{_2})$ are two $LGT$-spaces and $\phi : F_1 \rightarrow F_2$ is an arbitrary join preserve map. Then the followings are equivalent.
		\begin{enumerate}
			\item[(a)] $\phi$ is $ CLG $.
			\item[(b)] $\phi(\overline{a}) \leqslant \overline{\phi(a)}$, for each $ a \in F_1 $.
			\item[(c)] $ \overline{\phi_*(b)} \leqslant \phi_*(\overline{b}) $, for each $ b \in F_2 $.
		\end{enumerate}
		\label{closed and continuous}
	\end{Thm}
	\begin{proof}
		(a $\Rightarrow$ b). Since $ \phi(a) \leqslant \overline{\phi(a)} $, $ a \leqslant \phi_*(\overline{\phi(a)}) $. Since $ \overline{\phi(a)} \in \tau_{_2}^*$, $ \phi_*(\overline{\phi(a)}) \in \tau_{_2}^* $, so $ \overline{a} \leqslant \phi_*(\overline{\phi(a)}) $ and therefore $ \phi(\overline{a}) \leqslant \overline{\phi(a)} $.
		
		(b $\Rightarrow$ c). Since $ \phi_*(b) \leqslant \phi_*(b) $, $ \phi(\phi_*(b)) \leqslant b $, then by the assumption
		\[ \phi( \overline{\phi_*(b)}) \leqslant \overline{\phi(\phi_*(b))} \leqslant \overline{b}  \quad \Rightarrow \quad \overline{\phi_*(b)} \leqslant \phi_*(\overline{b})   \]
		
		(c $\Rightarrow$ a). For each $ t \in \tau_{_2}  $, $ \overline{\phi_*(t^*)} \leqslant \phi_*(\overline{t^*}) = \phi_*(t^*)  $, so $ \overline{\phi_*(t^*)} = \phi_*(t^*) $, hence $  \phi_*(t^*) \in \tau_{_1}^* $. Consequently, $ \phi $ is a $ CLG $ map.
	\end{proof}     
	
	\begin{Cor}
		Suppose that  $(F_1, \tau_{_1})$ and $(F_2,\tau_{_2})$ are two $LGT$-spaces and $\phi : F_1 \rightarrow F_2$ is an onto $ CLG $ map. If $ a \in F_1 $ is dense in $ F_1 $, then $ \phi(a) $ is dense in $ F_2 $.
	\end{Cor}
	\begin{proof}
		It follows easily from the above theorem.
	\end{proof}

	\begin{Thm}
		Suppose that $(F_1, \tau_{_1})$ and $(F_2,\tau_{_2})$ are two $LGT$-spaces and $\phi : F_1 \rightarrow F_2$ is an arbitrary join preserve map. Then $\phi$ is $ OLG $; \ff $ \phi_*(b^\circ) \leqslant \phi_*(b)^\circ $, for each $ b \in F_2 $.
		\label{open and continuous}
	\end{Thm}
	\begin{proof}
		$\Rightarrow$). Since $ b^\circ \in \tau_{_2} $, $ \phi_*(b^\circ) \in \tau_{_1} $, so 
		\[ \phi_*(b^\circ) = \phi_*(b^\circ)^\circ \leqslant \phi_*(b)^\circ \]
		
		$\Leftarrow$). For each $t \in \tau_{_2}$, $ \phi_*(t) = \phi_*(t^\circ) \leqslant \phi_*(t)^\circ $, so $ \phi_*(t) = \phi_*(t)^\circ \in \tau_{_1}  $, and therefore $ \phi $ is $ OLG $.
	\end{proof}
	
	\begin{Thm}
		Suppose that $(F_1, \tau_{_1})$, $(F_2,\tau_{_2})$ and $(F_3,\tau_{_3})$ are $LGT$-spaces. 
		\begin{itemize}
			\item[(a)] If $\phi : F_1 \rightarrow F_2 $ and $ \psi: F_2 \rightarrow F_3$ are $ OLG $ maps, then $\psi \circ \phi: F_1 \rightarrow F_3 $ is an $ OLG $ map.
			\item[(b)]If $\phi : F_1 \rightarrow F_2 $ and $ \psi: F_2 \rightarrow F_3$ are $ CLG $ maps, then $\psi \circ \phi: F_1 \rightarrow F_3 $ is a $ CLG $ map.
			\item[(c)]If $\phi : F_1 \rightarrow F_2 $ and $ \psi: F_2 \rightarrow F_3$ are $ LG $ maps, then $\psi \circ \phi: F_1 \rightarrow F_3 $ is an $ LG $ map.
		\end{itemize}
		\label{composition of continuous}
	\end{Thm} 
	\begin{proof}
		(a). By Proposition \ref{right adjoint} and Theorem \ref{open and continuous}, for every $ b \in F_2 $, we have 
		\[ ( \phi \circ \psi)_* (a^\circ) =  \psi_*(\phi_*(a^\circ))  \leqslant \big(\psi_*(\phi_*(a))\big)^\circ = \big( ( \phi \circ \psi)_* (a) \big)^\circ \]
		Hence, by Theorem \ref{open and continuous}, $ \psi \circ \phi $ is an $ OLG $ map.
		
		(b). By Theorem \ref{closed and continuous}, for every $ a \in F_1 $, we have $ \psi(\phi(\overline{a})) \leqslant \psi(\overline{\phi(a)}) \leqslant \overline{\psi(\phi(a))} $. Hence, by Theorem \ref{closed and continuous}, $ \psi \circ \phi $ is a $ CLG $ map.
		
		(c). It follows immediately from (a) and (b).
	\end{proof}
	
	In the following proposition, we show that the $ CLG $, $ OLG $ and $ LG $ maps are extensions of the continuous function.
	
	Suppose that $ f : X \rightarrow Y $; we put $ \mho(f) : \mathcal{P}(X) \rightarrow \mathcal{P}(Y)  $, defined by $ \mho(f) (S) = f(S) $, for each $ S \in \mathcal{P}(X) $. 
	
	\begin{Pro}
		If $ (X,\tau) $ and $ (Y,\tau') $ are topological spaces and $f:X \rightarrow Y $, then the followings are equivalent:
		\begin{itemize}
			\item[(a)] $\mho(f)$ is an $ OLG $ map.
			\item[(b)] $\mho(f)$ is a $ CLG $ map.
			\item[(c)] $\mho(f)$ is an $ LG $ map.
			\item[(d)] $f$ a is continuous function.
		\end{itemize}
		\label{extension of continuous}
	\end{Pro}
	\begin{proof}
		Since $\mho(f)_* (S) = \bigvee_{\mho(f)(A) \subseteq S} A = \bigcup_{f(A) \subseteq S} A = f^{-1}(S)  $, the above statements are equivalent.
	\end{proof}
	
	For each topological space $ (X,\tau) $, we put $ \mho(X,\tau) = (\mathcal{P}(X),\tau) $. Now we can conclude the following corollary from Theorem \ref{composition of continuous},  Proposition \ref{extension of continuous} and this fact that $ \mho(I_{_X}) = I_{_{\mathcal{P}(X)}} = I_{_{\mho(X)}} $.
	
	\begin{Cor}
		If $ \mathbf{Top} $ is the category of topological spaces and continuous maps and $ \mathbf{Lgt} $ is the category of $ LGT $-spaces and $ LG $ maps, then $ \mho : \mathbf{Top} \rightarrow \mathbf{Lgt} $ is a functor.
	\end{Cor}
	
	\section{Subspace}
	
	This section has been devoted to studying the relations between the $ OLG $, $ CLG $ and $ LG $ maps and the subspaces. It has been shown some versions of the relations between a continuous function and a subspace satisfy for these extension. Also, it has been gave some counterexamples for some general versions which do not satisfy. 
	
	\begin{Lem}
		Suppose that $(F, \tau)$ and $(F',\tau')$ are two $LGT$-spaces, $ a \in F $ and $\phi : (F,\tau) \rightarrow (F',\tau')$. Then $ (\phi|_{_{F_a}})_* (y) = \phi_*(y) \wedge a $, for each $ y \in F' $.
		\label{adjoint restriction}
	\end{Lem}
	\begin{proof}
		For each $ x \in F_a $, $ x \leqslant a  $, then 
		\begin{align*}
		x \leqslant (\phi|_{_{F_a}})_* (y)  & \quad \Leftrightarrow \quad \phi|_{_{F_a}}(x) \leqslant y \quad \Leftrightarrow \quad  \phi(x) \leqslant y  \\
		& \quad \Leftrightarrow \quad x \leqslant \phi_*(y) \quad \Leftrightarrow \quad x \leqslant \phi_*(y)  \wedge a 
		\end{align*}
		Hence $ (\phi|_{_{F_a}})_* (y) = \phi_*(y) \wedge a $.	
	\end{proof}
	
	\begin{Thm}
		Suppose that $(F, \tau)$ and $(F',\tau')$ are two $LGT$-spaces, $ a \in F $ and $\phi : (F,\tau) \rightarrow (F',\tau')$.
		\begin{itemize}
			\item[(a)] 	If $ \phi $ is an $ OLG $ map, then $\phi|_{_{F_a}} : (F_a,\tau_a) \rightarrow (F',\tau')$ is an $ OLG $ map.
			\item[(b)] 	If $ a \in F^* $ and $ \phi $ is a $ CLG $ map, then $\phi|_{_{F_a}} : (F_a,\tau_a) \rightarrow (F',\tau')$ is a $ CLG $ map.
			\item[(c)] 	If $ a \in F^* $ and $ \phi $ is an $ LG $ map, then $\phi|_{_{F_a}} : (F_a,\tau_a) \rightarrow (F',\tau')$ is an $ LG $ map.
		\end{itemize}
		\label{restriction of continuous}
	\end{Thm} 
	\begin{proof}
		(a). Suppose that $ t \in \tau' $, then $ \phi_*(t) \in \tau $, so by Lemma \ref{adjoint restriction}, $ (\phi|_{_{F_a}})_* (t) = \phi_*(t) \wedge a  $, hence $ (\phi|_{_{F_a}})_* (t)  \in F_a $. Consequently, $ \phi|_{_{F_a}} $ is $ OLG $. 
		
		(b). By \cite[Proposition 3.4]{aliabad2015lg}, the proof is similar to the proof of part (a).
		
		(c). It follows immediately from parts (a) and (b).
	\end{proof}
	
	Clearly, since each element of a complemented frame is a complement of some element, the above theorem implies the following corollary.
	
	\begin{Cor}
		Suppose that $(F, \tau)$ and $(F',\tau')$ are two $LGT$-spaces, $ a \in F $ and $\phi : (F,\tau) \rightarrow (F',\tau')$. If $ \phi $ is an $ LG $ map and $ F $ is complemented, then $\phi|_{_{F_a}} : (F_a,\tau_a) \rightarrow (F',\tau')$ is an $ LG $ map.
	\end{Cor}
	
	\begin{Lem}
		Suppose that $(F,\tau)$ is an $ LGT $-space and $a \in \tau$.
		\begin{itemize}
			\item[(a)] For each $t \in F_a$, $t \in \tau_a$, \ff  $t \in \tau$.
			\item[(b)] For each $b \in F_{a^*}$, $b \in \tau_{a^*}$, \ff $b \in \tau^*$. 
		\end{itemize}
		\label{open and closed element in open and closed subspace}
	\end{Lem}
	\begin{proof}
		(a). It is straightforward.
		
		(b). By \cite[Proposition 3.4]{aliabad2015lg}, the proof is similar to part (a).	
	\end{proof} 
	
	\begin{Thm}
		Suppose that $ \phi : F \rightarrow F' $ is an arbitrary join preserve map.
		\begin{itemize}
			\item[(a)] If $(F, \tau)$ and $(F',\tau')$ are two $LGT$-spaces, $ s , t \in \tau_{_1} $, $ \phi|_{F_s} $ and $ \phi|_{F_t} $ are $ OLG $ maps and $ s \vee t = 1 $, then $ \phi $ is an $ OLG $ map.
			\item[(b)] If $(F, \tau)$ and $(F',\tau')$ are two $LT$-spaces, $ a , b \in \tau_{_1}^* $, $ \phi|_{F_a} $ and $ \phi|_{F_b} $ are $ CLG $ maps and $ a \vee b = 1 $, then $ \phi $ is a $ CLG $ map.
		\end{itemize}
	\end{Thm}
	\begin{proof}
		(a). Suppose that $ u \in \tau' $. Then $ (\phi|_{_{F_s}})_* (u) , (\phi|_{_{F_t}})_*(u) \in \tau $, thus, by Lemma \ref{adjoint restriction},
		\begin{align*}
		\phi_*(u)  & = \phi_*(u) \wedge ( s \vee t ) \\
		& = ( \phi_*(u) \wedge s ) \vee ( \phi_*(u) \wedge t ) \\
		& = (\phi|_{_{F_s}})_*(u) \vee (\phi|_{_{F_t}})_*(u) \in \tau
		\end{align*}
		Hence $ \phi $ is $ OLG $.
		
		(b). Since $(F,\tau)$ and $(F',\tau')$ are two $LT$-spaces, the proof is similar to the proof of part (a).
	\end{proof}
	
	\begin{Pro}
		Suppose that $ (F,\tau) $ and $ (F',\tau') $ are two $ LGT $-spaces. If $ \phi : F' \rightarrow F $ is an $ OLG $ map such that $ \phi(F') = F_a $, for some $ a \in F $, then $ \phi : F \rightarrow F_a $ is $ OLG $.
	\end{Pro}
	\begin{proof}
		Suppose that $ s \in \tau_a $, then $ t \in \tau $ exists such that $ s = t \wedge a $. Since $ \phi(1) = a $, $ \phi_*(a) = 1 $, so 
		\[  \phi_*(s) = \phi_*(t \wedge a) = \phi_*(t) \wedge \phi_*(a) = \phi_*(t) \wedge 1 = \phi_*(t) \in \tau' \] 
		Hence $ \phi : F' \rightarrow F_a  $ is $ OLG $.
	\end{proof}
	
	In the following example we show that the $ LG $ map image of a compact $ LGT $-space need not be compact.
	
	\begin{Exa}
		Suppose that $ F_2 = [0,1] $ with ordinary relation and $ F_1 = [0,\frac{1}{2}] \cup \{ a , b , 1 \} $ with the following illustrated relation 
		
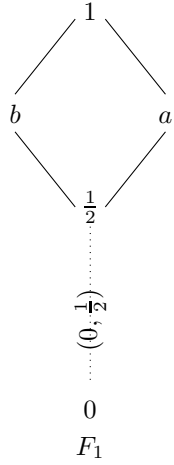
\begin{center}
			\begin{tikzpicture}
			\draw(3,-1.2)node{$F_1$};
			\draw(3,-0.7)node{$0$};
			\draw(3,0.5)node[rotate=90]{$(0,\frac{1}{2})$};
			\draw[dotted](3,-0.3)--(3,1.8);
			\draw(3,2)node{$\frac{1}{2}$};
			\draw(3.2,2)--(4,3)node[above]{$a$};
			\draw(2.8,2)--(2,3)node[above]{$b$};
			\draw(4,3.5)--(3.2,4.5);
			\draw(2,3.5)--(2.8,4.5);
			\draw(3,4.6)node{$1$};
			\end{tikzpicture}
		\end{center}
		It is easy to check that $ F_1 $ and $ F_2 $ are frames. Consider $ \tau_{_1} = F_1 $ and $ \tau_{_2} = F_2 $. Clearly $ \tau_{_1}^* = \{‌0,1,a,b\} $, $ \tau_{_2}^* = \{ 0 ,1 \} $,  $ (F_1,\tau_{_1}) $ is a compact $ LGT $-space, $ (F_2,\tau_{_2}) $ is not a compact $ LGT $-space and the map $ \phi : (F_1,\tau_{_1}) \rightarrow (F_2,\tau_{_2}) $, defined by
		\[ \phi(x) = \begin{cases}
		1 & x \notin [0,\frac{1}{2}] \\
		2x & x \in [0,\frac{1}{2}]
		\end{cases} \]
		is an onto $ LG $ map.
	\end{Exa}
	
	\begin{Pro}
		Suppose that $ (F_1,\tau_{_1}) $ and $ (F_2,\tau_{_2}) $ are two $ LGT $-spaces, $ \phi : F_1 \rightarrow F_2 $ is an onto $ OLG $ and $ \bigvee_{\alpha \in A} u_\alpha = 1 $ implies that $ \bigvee_{\alpha \in A} \phi_*(u_\alpha) = 1 $, for every subfamily $ \{ u_\alpha \}_{\alpha \in A}  $ of $ \tau_{_2} $. 
		\begin{itemize}
			\item[(a)] If $ F_1 $ is compact, then $ F_2 $ is compact.
			\item[(b)] If $ F_1 $ is countably compact, then $ F_2 $ is countably compact.
			\item[(c)] If $ F_1 $ is Lindel\"of , then $ F_2 $ is Lindelo\"f.
		\end{itemize}
		\label{image of compact}
	\end{Pro}
	\begin{proof}
		(a). Suppose that $ \{ u_\alpha \}_{\alpha \in A} $ is a family of open elements of $ F_2 $ and $ \bigvee_{ \alpha \in A } u_\alpha = 1 $. Since $ \phi $ is an $ OLG $ map, $ t_\alpha  = \phi_*(u_\alpha ) \in \tau_\alpha  $. By the assumption, $ \bigvee_{ \alpha \in A } t_\alpha = 1 $. Since $ F_1 $ is compact, there are $ \alpha_1, \alpha_2, \ldots, \alpha_n \in A  $ such that $ 1 = \bigvee_{i=1}^n  t_{\alpha_i} $, so 
		\[ 1 = \phi\Big(‌\bigvee_{i=1}^n t_{\alpha_i}  \Big) = \bigvee_{i=1}^n \phi(t_{\alpha_i}) = \bigvee_{i=1}^n u_{\alpha_i}. \]
		Consequently, $ F_2 $ is compact.
		
		(b) and (c). They are similar to (a).
	\end{proof}
	
	Suppose that, $ (‌F_1,\tau_{_1}) $ and $ (F_2,\tau_{_2}) $ are two frames, $ \phi : F_1 \rightarrow F_2 $ is an arbitrary join preserve map and $ B $ is a base for $ \tau_{_2} $. In  the following example we show that if $ \phi_*(b) \in \tau_{_1} $,  for each $ b \in B $, then $ \phi $ need not be $ OLG $.  
	
	\begin{Exa}
		Suppose that 
		
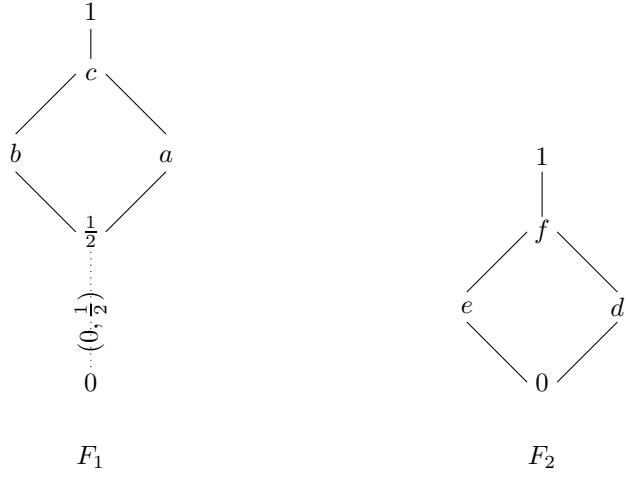
\begin{center}
			\begin{tikzpicture}
			\draw(3,-1)node{$F_1$};
			\draw(3,0)node{$0$};
			\draw(3,0.8)node[rotate=90]{$(0,\frac{1}{2})$};
			\draw[dotted](3,0.2)--(3,1.8);
			\draw(3,2)node{$\frac{1}{2}$};
			\draw(3.2,2)--(4,2.8)node[above]{$a$};
			\draw(2.8,2)--(2,2.8)node[above]{$b$};
			\draw(4,3.3)--(3.2,4.1);
			\draw(2,3.3)--(2.8,4.1);
			\draw(3,4.1)node{$c$};
			\draw(3,4.3)--(3,4.7)node[above]{$1$};
			
			\draw(9,-1)node{$F_2$};
			\draw(9,0)node{$0$};
			\draw(8.8,0)--(8,0.8);
			\draw(9.2,0)--(10,0.8);
			\draw(10,1)node{$d$};
			\draw(8,1)node{$e$};
			\draw(8.8,2)--(8,1.2);
			\draw(9.2,2)--(10,1.2);
			\draw(9,2)node{$f$};
			\draw(9,2.2)--(9,2.8);
			\draw(9,3)node{$1$};
			\end{tikzpicture}
		\end{center}
		$ \tau_{_2} = F_2 $, $ \tau_{_1} = [0,\frac{1}{2}]  \cup \{ 1 \} $, $ B = \{ 0,d,e,1 \} $ and $ \phi : F_1 \rightarrow F_2 $ is defined by $ \phi|_{[0,\frac{1}{2}]} = 0 $, $ \phi(1) = 1 $ and $  \phi(a) = \phi(b) = \phi(c) = f $. Then $ \phi_*(1) = c \notin \tau_{_1} $, so $ \phi $ is not $ OLG $, but $ \phi_*(x) \in \tau_{_1} $, for each $ x \in B $.
	\end{Exa}
	
	We finish this section by a proposition, in which, by adding a condition; we show that  if $ \phi_*(b) $ is an open element, for each element $ b $ of a base, then $ \phi $ is $ OLG $.
	
	\begin{Pro}
		Suppose that $ (‌F_1,\tau_{_1}) $ and $ (F_2,\tau_{_2}) $ are two frame, $ \phi : F_1 \rightarrow F_2 $ an arbitrary join preserve map, $ B $ is a base for $ \tau_{_2} $ and $ \phi_*(\bigvee_{\alpha \in A} b_\alpha) = \bigvee_{\alpha \in A} \phi_*(b_\alpha) $, for each subfamily $ \{ b_\alpha \}_{\alpha \in A} $ of $ B $. If $ \phi_*(b) \in \tau_{_1} $, for each $ b \in B $, then $ \phi $ is $ OLG $.
		\label{base and continuous}
	\end{Pro}
	\begin{proof}
		It is straightforward.
	\end{proof}
	
	\section{Product and Quotient}
	
	In this section, we study some related items to product $ LGT $-spaces. Then by inspired of the well-known concepts in the topology literature, the new concepts, $ LGT $-space generated by some family of maps, quotient $ LGT $-space and decomposition topology have been introduced and studied.
	
	\begin{Thm}
		Let  $ \{ (‌F_\alpha,\tau_\alpha) \}_{\alpha \in A} $ be a family of $ LGT $-spaces and  $ (‌ \prod_{\alpha \in A} F_\alpha, \tau_{_P} ) $ be the product $ LGT $-space. Then for each $ \beta \in A $, the projection map $ \pi_{_\beta} :  \prod_{\alpha \in A} F_\alpha \rightarrow F_\beta $ is an $ OLG $ map.
		\label{restriction of projection}
	\end{Thm}
	\begin{proof}
		Clearly $ \pi_{_\beta} $ is an arbitrary join preserve map. Suppose that $ t_\beta \in \tau_{_\beta}  $, pick $ v \in \prod_{\alpha \in A} F_\alpha $, in which 
		\[ v_\alpha = \begin{cases}
		1 & \alpha \neq \beta \\
		t_\beta & \alpha = \beta 
		\end{cases}. \]
		For each $ x \in \prod_{\alpha \in A} F_\alpha $,
		\[ x \leqslant ( \pi_{_\beta} )_*(t_\beta) \quad \Leftrightarrow \quad \pi_{_\beta} (x) \leqslant t_\beta \quad \Leftrightarrow \quad x_\beta \leqslant t_\beta \quad \Leftrightarrow \quad x \leqslant v \]
		Hence $ (\pi_{_\beta})_*(t_\beta) = v \in \tau_{_P} $.
	\end{proof}
	
	\begin{Cor}
		Let $ \{ (F_\alpha,\tau_\alpha) \}_{\alpha \in A} $ be a family of $ LGT $-spaces, $ ( \prod_{\alpha \in A} F_\alpha , \tau_{_P} ) $ be the product $ LGT $-space and $ (F,\tau) $ is an $ LGT $-space. An arbitrary join preserve map $ \phi : F \rightarrow \prod_{\alpha \in A} F_\alpha  $ is $ OLG $; \ff $ \pi_{_\alpha} \circ \phi  $ is $ OLG $, for every $ \alpha \in A $. 
		\label{contiuous on product}
	\end{Cor}  
	\begin{proof}
		($ \Rightarrow $). It follows immediately from Theorems \ref{composition of continuous} and \ref{restriction of projection}.
		
		($ \Leftarrow $). Suppose that $ t \in \tau_{_P} $. It easy to see that, for some $ n \in \mathbb{N} $, we have $ t = \bigwedge_{i=1}^n ( \pi_{_{\alpha_i}} )_* (t_{\alpha_i}) $, in which $ \alpha_i \in A $ and $ t_{\alpha_i} \in \tau_{_{\alpha_i}} $, for every $ 1 \leqslant i \leqslant n $. Since $ \pi_{\alpha_i} \circ \phi $ is $ OLG $, $ ( \pi_{\alpha_i} \circ \phi )_* (t_{\alpha_i}) \in \tau' $, for $ 1 \leqslant i \leqslant n $. Then, by Proposition \ref{right adjoint},	
		\[ \phi_*(t) =  \phi_* \Big( \bigwedge_{i=1}^n (\pi_{\alpha_i})_* (t_{\alpha_i}) \Big) = \bigwedge_{i=1}^n \phi_* (‌\pi_{\alpha_i*}(t_{\alpha_i})) = \bigwedge_{i=1}^n (‌\pi_{\alpha_i} \circ \phi)_* (t_{\alpha_i}) \in \tau'  \]
		Hence $ \phi $ is $ OLG $.
	\end{proof}
	
	\begin{Def}
		Suppose that $ \{ (F_\alpha,\tau_\alpha) \}_{\alpha \in A} $ is a family of $ LGT $-spaces, $ F $ is a frame and $ \phi_\alpha :‌F \rightarrow F_\alpha $ is an arbitrary join preserve map, for each $ \alpha \in A $. The $LG$-topology generated by the family $ \{‌ \phi_{\alpha*}(t_\alpha) : \alpha \in A \text{ and } t_\alpha \in \tau_\alpha \} $ is called the \emph{weak $LG$-topology generated} by $ \{ \phi_\alpha \}_{\alpha \in A} $ and $ F $ with this topology is called \emph{weak $ LGT $-space generated} by $ \{ \phi_\alpha \}_{\alpha \in A} $.
	\end{Def}
	
	Actually, the product $LG$-topology is not a good extension of the product topology. But the weak $LG$-topology generated by the family $ \{ \mho(\pi_{_\alpha}) \}_{\alpha \in A} $ on $ \mathcal{P}(\prod_{\alpha \in A} X_\alpha) $ coincides with $ \mho( \prod_{\alpha \in A} X_\alpha , \tau) $, in which $ \tau $ is the product topology.
	
	\begin{Thm}
		Suppose that $ \{ (F_\alpha,\tau_\alpha) \}_{\alpha \in A} $ is a family of $ LGT $-spaces, $ (F,\tau) $ is the weak $ LGT $-space generated by $ \{ \phi_\alpha  :‌ F \rightarrow F_\alpha  \}_{\alpha \in A} $, $ (F',\tau') $ is an $ LGT $-space and $ \psi : F' \rightarrow F $. If $ \psi $ and $ \psi_* $ are an arbitrary join preserve maps; then $ \psi $ is $ OLG $, \ff $ \phi_\alpha \circ \psi $ is $ OLG $, for each $ \alpha \in A $.
	\end{Thm}
	\begin{proof}
		By Proposition \ref{base and continuous}, it similar to the proof of Corollary \ref{contiuous on product}.
	\end{proof}
	
	\begin{Pro}
		Suppose that $ (F',\tau') $ is an $ LGT $-space and $ F  $ is a frame. If $ \phi : F' \rightarrow F $ is onto and $ \phi $ and $ \phi_* $ are arbitrary join preserve maps, then $ \tau_\phi = \{ t \in F : \phi_*(t) \in \tau_{_1} \} $ is the greatest  $LG$-topology on $ F $, where $ \phi $ is $ OLG $. 
		\label{quotient}
	\end{Pro}
	\begin{proof}
		Since $ \phi_* $ is an arbitrary join preserve map, $ \phi_*(0) = 0 $ and therefore $ 0 \in \tau_\phi $. Now, suppose that $ \phi(a) = 1 $, for some $ a \in F $, then 
		\[ \phi(1) = \phi( a \vee 1 )‌ = \phi (a) \vee \phi(1) = 1 \]
		Hence $ \phi_*(1) = \bigvee_{\phi(x) \leqslant 1} x = 1 $, so $ 1 \in \tau_\phi $. Since $ \phi_*\big( \bigvee_{\alpha \in A} t_\alpha \big) = \bigvee_{\alpha \in A} \phi_*(t_\alpha) $ and $ \phi_*(t_1 \wedge t_2) = \phi_*(t_1)‌ \wedge \phi_*(t_2) $, for every subfamily $ \{ t_\alpha \}_{\alpha \in A} $ of $ \tau' $ and $ t_1 $ and $ t_2 $ in $ \tau' $, $  \bigvee_{\alpha \in A} t_\alpha \in \tau_\phi $ and $ t_1 \wedge t_2 \in \tau_\phi $. Consequently, $ \tau_\phi $ is an $LG$-topology on $ F $. Clearly, $ \tau_\phi $ is the greatest  $LG$-topology on $ F $, where $ \phi $ is $ OLG $.
	\end{proof}
	
	\begin{Def}
		Suppose that $ (F',\tau') $ is an $ LGT $-space and $ F  $ is a frame. If $ \phi : F' \rightarrow F $ is onto and $ \phi $ and $ \phi_* $ are arbitrary join preserve maps. Then by the above proposition, the $LG$-topology $ \tau_\phi = \{ t \in F : \phi_*(t) \in \tau'  \} $ is called the \emph{quotient $LG$-topology} on $ F $ induced by $ \phi $.
	\end{Def}
	
	\begin{Thm}
		Suppose that $ (F',\tau') $ and $ (F'',\tau'') $ are two $ LGT $-spaces and $ (F,\tau_\phi) $ is the quotient $LG$-topology induced by $ \phi : F' \rightarrow F $. An arbitrary join preserve map $ \psi : F \rightarrow F''  $ is $ OLG $, \ff $ \psi \circ \phi : F' \rightarrow F'' $ is $ OLG $.
	\end{Thm}
	\begin{proof}
		$ \Rightarrow $). It is trivial, by Theorem \ref{composition of continuous}.
		
		$ \Leftarrow $). If $ t \in \tau'' $, then  $ ( \psi \circ \phi )_* (t) \in \tau' $, so $ \phi_* ( \psi_*(t) ) \in \tau'  $, by Proposition \ref{right adjoint}.  Thus $ \psi_*(t) \in \tau_\phi $, and therefore $ \psi $ is $ OLG $. 
	\end{proof}
	
	\begin{Def}
		Suppose that $ F $ is a frame. $ D \subseteq F $ is called a partition for $ F $, if 
		\begin{itemize}
			\item[(a)] $ 0 \notin D $.
			\item[(b)] $ \bigvee D = 1 $.
			\item[(c)] For every $ d_1,d_2 \in D $, $ d_1 \neq d_2 $ implies that $ d_1 \wedge d_2 = 0 $.
		\end{itemize}
	\end{Def}
	
	\begin{Lem}
		Suppose that $ F $ is a frame, $ a \in F $, $ D \subseteq F $ is a partition for $ F $, $ T \subseteq D $ and $ T_a = \{ d \in D : d \wedge a \neq 0 \} $. Then $ a \leqslant \bigvee T  $, \ff $ T_a \subseteq T $.
		\label{elements in partition}
	\end{Lem}
	\begin{proof} If $ a = 0 $, then it is clear. Now suppose that $ a \neq 0 $.
		
		($\Rightarrow$). On contrary, suppose that $ d' \in T_a \setminus T $ exists, then 
		\[a = a \wedge \big( \bigvee_{d \in T} d \big) = \bigvee_{d \in T} ( d \wedge a )  \; \Rightarrow \; 0 \neq d' \wedge a = d' \wedge \Big[ \bigvee_{d \in T} ( d \wedge a ) \Big] = \bigvee_{d \in T} ( d' \wedge d \wedge a) = 0 \]
		which is a contradiction.
		
		($\Leftarrow$). Since for each $ d \in D \setminus T $, $ d \wedge a = 0 $,
		\begin{align*}
		a  = a \wedge 1 & = a \wedge \Big( \bigvee_{d \in D} d \Big) = \bigvee_{d \in D} ( a \wedge d ) \\
		& = \bigg[ \bigvee_{d \in T} ( a \wedge d ) \vee \bigvee_{d \in D \setminus T} ( a \wedge d ) \bigg] \\ 
		& = \bigvee_{d \in T} ( a \wedge d ) = a \wedge \Big( \bigvee_{d \in T} d \Big).  
		\end{align*} 
		Consequently, $ a \leqslant \bigvee_{d \in T} d $.
	\end{proof}
	
	\begin{Pro}
		Suppose that $ (F,\tau) $ is an $ LGT $-space. If $ D $ is a partition for $ F $, then $ \tau_{_D} = \{ T \subseteq D : \bigvee T \in \tau \} $ is a topology on $ D $.
	\end{Pro}
	\begin{proof}
		Clearly $ \bigvee \emptyset = \emptyset $ and, by the assumption, $ \bigvee D = 1 $, so $ \emptyset , D \in \tau_{_D} $.
		
		Suppose that $ T_1,T_2 \in \tau_{_D} $. By Lemma \ref{elements in partition},
		\begin{align*}
		a \leqslant \bigvee_{d \in T_1 \cap T_2 } d \quad 
		& \Leftrightarrow \quad T_a \subseteq T_1 \cap T_2 \quad \Leftrightarrow \quad T_a \subseteq T_1 \quad \text{ and } \quad T_a \subseteq T_2 \\
		& \Leftrightarrow \quad a \leqslant \bigvee_{d \in T_1} d \quad \text{ and } \quad  a \leqslant \bigvee_{d \in T_2 } d \\
		& \Leftrightarrow \quad a \leqslant \Big( \bigvee_{d \in T_1} d \Big) \wedge \Big( \bigvee_{d \in T_2 } d \Big).
		\end{align*}
		Thus $ \bigvee_{d \in T_1 \cap T_2 } d = \Big( \bigvee_{d \in T_1} d \Big) \wedge \Big( \bigvee_{d \in T_2 } d \Big) \in \tau $, and therefore $ T_1 \cap T_2 \in \tau_{_D} $.
		
		Now suppose that $ \{ T_\alpha \}_{\alpha \in A} \subseteq \tau_{_D} $. Since $ \bigvee_{d \in \bigcup_{\alpha \in A} T_\alpha } d = \bigvee_{\alpha \in A} \bigvee_{d \in T_\alpha } d \in \tau $, $ \bigcup_{\alpha \in A} T_\alpha \in \tau_{_D} $. Consequently, $ (D,\tau_{_D}) $ is a topological space.
	\end{proof}
	
	\begin{Def}
		Suppose that $ (F,\tau) $ is an $ LGT $-space and $ D $ is a partition for $ F $. By the above proposition, $ \tau_{_D} = \{ T \subseteq D : \bigvee T \in \tau \} $ is a topology on $ D $, $ \tau_{_D} $ is called the decomposition topology and $ (\mathcal{P}(D),\tau_{_D}) $ is called decomposition $LGT$-space.
	\end{Def}

	\begin{Thm}
		Suppose that $ (F,\tau) $ is an $ LGT $-space and $ D $ is a partition for $ F $. The decomposition topology $ \tau_{_D} $ is a quotient $LG$-topology on $ \mathcal{P}(D) $. 
	\end{Thm}
	\begin{proof}
		Set $ P : F \rightarrow \mathcal{P}(D) $, defined by $ P(a) = \{ d \in D : d \wedge a \neq 0 \} = T_a $. By Lemma \ref{elements in partition}, $ a \leqslant \bigvee T $, \ff $ T_a \subseteq T $; \ff $ P(a) \subseteq T $ and this is equivalent to say that $ a \leqslant P_*(T) $. Consequently, $ P_*(T) = \bigvee T \quad (*)$.
		
		Clearly, $ P(0) = \emptyset  $. Suppose that $ \{ T_\alpha \}_{\alpha \in A } \subseteq P(D) $, then by $(*)$,
		\[ P_*\Big(\bigcup_{\alpha \in A} T_\alpha \Big) = \bigvee_{\alpha \in A} \Big( \bigvee T_\alpha \Big) = \bigvee_{\alpha \in A} P_*(T_\alpha). \]
		Thus $ P_* $ is arbitrary join preserve. 
		
		Now suppose that $ \{ a_\alpha \}_{\alpha \in A}  $. Then for each $ T \subseteq D $,
		\begin{align*}
		P\big(\bigvee_{\alpha \in A} a_\alpha\big) \subseteq T \quad 
		& \Leftrightarrow \quad \bigvee_{\alpha \in A} a_\alpha \leqslant P_*(T) & \\
		& \Leftrightarrow \quad \bigvee_{ \alpha \in A } a_\alpha \leqslant \bigvee T & \text{ By } (*) \\
		& \Leftrightarrow \quad \forall \alpha \in A \quad a_\alpha \leqslant \bigvee T & \\
		& \Leftrightarrow \quad \forall \alpha \in A \quad T_{a_\alpha} \subseteq T & \text{By Lemma \ref{elements in partition}} \\
		& \Leftrightarrow \quad \bigcup_{\alpha \in A} T_{a_\alpha} \subseteq T & 
		\end{align*}
		Thus $ P\big(\bigvee_{\alpha \in A} a_\alpha\big) = \bigcup_{\alpha \in A} T_{a_\alpha} = \bigcup_{\alpha \in A} P(a_\alpha) $, hence $ P $ is arbitrary join preserve.
		
		Finally, by $(*)$,
		\[ T \in \tau_{_D} \quad \Leftrightarrow \quad P_*(T) = \bigvee T \in \tau \quad \Leftrightarrow \quad T \in \tau_{_P} \]
		Hence $ \tau_{_D} = \tau_{_P} $.
	\end{proof}
	
	Clearly, a non-topological $ LGT $-space $ (F,\tau) $ exists. Then $ \tau_{_{I_F}} = \tau $, in which $ I_F $ is the identity map, hence $ \tau_{_{I_F}} $ is a quotient $LG$-topology which is not decomposition topology. Therefore the converse of the above theorem is not true in generally.
	
	\section{Isomorphism}
	
	In the last section, first we introduce and study open and closed map and then, by the use of the $ LG $ map, we introduce an isomorphism, called $ LG $ map, between $ LGT $-spaces. Finally some $ LG $-properties have been studied.
	
	\begin{Def}
		Suppose that $ (F_1, \tau_{_1}) $ and $ (F_2,\tau_{_2} ) $ are two $ LGT $-spaces. A map $ \phi : F_1 \rightarrow F_2 $ is called open (closed) map if $ \phi(t) \in \tau_{_2} $ ($ \phi(t^*) \in \tau_{_2}^* $), for each $ t \in \tau_{_1} $.  
	\end{Def}
	
	\begin{Pro}
		Suppose that $ (F_1, \tau_{_1}) $ and $ (F_2,\tau_{_2} ) $ are two $ LGT $-spaces and $ \phi : F_1 \rightarrow F_2 $. If $ \phi $ and $ \phi_* $ are arbitrary join preserve maps, then
		\begin{itemize}
			\item[(a)] $ \phi $ is an open map, \ff $ \phi_* $ is an $ OLG $ map.
			\item[(b)]  $ \phi $ is a closed map, \ff $ \phi_* $ is a $ CLG $ map.
		\end{itemize}
		\label{open and and olg right adjoin}
	\end{Pro}
	\begin{proof}
		It is clear, by Lemma \ref{preserving and right adjoin}.
	\end{proof}
	
	\begin{Pro}
		Let  $ \{ (‌F_\alpha,\tau_\alpha) \}_{\alpha \in A} $ be a family of $ LGT $-spaces and  $ (‌ \prod_{\alpha \in A} F_\alpha, \tau_{_P} ) $ be the product $ LGT $-space. Then for each $ \beta \in A $, the projection map $ \pi_{_\beta} :  \prod_{\alpha \in A} F_\alpha \rightarrow F_\beta $ is an open map.
	\end{Pro}
	\begin{proof}
		It is straightforward.
	\end{proof}
	
	\begin{Pro}
		Suppose that $ (F_1,\tau_{_1}) $ and $ (F_2,\tau_{_2}) $ are two $ LGT $-spaces, $ \phi : F_1 \rightarrow F_2 $, $ \phi $ and $ \phi_* $ are two arbitrary join preserve and $ \phi(0) = 0 $. If $ \phi $ is an onto open $ OLG $, then $ \tau_{_2} = \tau_\phi $.
	\end{Pro}
	\begin{proof}
		It is easy, by Propositions \ref{right adjoint} and \ref{quotient}.
	\end{proof}
	
	\begin{Lem}
		Suppose that $ (F_1,\tau_{_1}) $ and $ (F_2,\tau_{_2}) $ are two $ LGT $-spaces. If $ \phi $ is an arbitrary join preserve one-to-one map from $ F_1 $ onto $ F_2 $, then
		\begin{itemize}
			\item[(a)] $ \phi(1) = 1 $.
			\item[(b)] $ \phi_*(0) = 0 $.
			\item[(c)] $ \phi_*(b^*) = ( \phi_* (b) )^* $, for every $ b \in F_2 $.
			\item[(d)] $ \phi(a^*) = ( \phi(a) )^* $, for every $ a \in F_1 $.
		\end{itemize}
		\label{one-to-one and complement}
	\end{Lem}
	\begin{proof}
		(a) and (b). They follow from Lemma  \ref{one-to-one and preserving}.
		
		(c). By Lemma \ref{one-to-one and preserving},
		\[ \phi(a^*) \wedge \phi(a) = \phi(a^* \wedge a) = \phi(0) = 0 \]
		If $ b \wedge \phi(a) = 0 $, for some $ b \in F_2 $, then $ a' \in F_1 $ exists such that $ \phi(a') = b $, so, by Lemma \ref{one-to-one and preserving},
		\[ 0= \phi(a') \wedge \phi(a) = \phi(a' \wedge a) \quad \Rightarrow \quad a' \wedge a = 0 \quad \Rightarrow \quad a' \leqslant a^* \quad \Rightarrow \quad b = \phi(a') \leqslant \phi(a^*) \]
		Hence $ \phi(a^*) = ( \phi(a) )^* $.
		
		(d). By Lemma  \ref{one-to-one and preserving}, it is similar that (c).
	\end{proof}
	
	Now we can conclude the following corollary form Propositions \ref{OLG is LG} and \ref{open and and olg right adjoin} and the above lemma.
	
	\begin{Cor}
		Suppose that $ (F_1,\tau_{_1}) $ and $ (F_2,\tau_{_2}) $ are two $ LGT $-spaces and $ \phi $ is a join preserve one-to-one map from $ F_1 $ onto $ F_2 $. Then the following statements are equivalent
		\begin{itemize}
			\item[(a)]  $ \phi $ is an $ OLG $ map.
			\item[(b)] $ \phi $ is an $ LG $ map
			\item[(c)] $ \phi_* $ is an open map.
		\end{itemize}
		\label{two sided and morphism}
	\end{Cor}
	
	Example \ref{clg is not olg} shows that an arbitrary join preserve one-to-one onto $ CLG $ map need not be $ LG $ map.
	
	\begin{Def}
		Suppose that $ (F_1, \tau_{_1}) $ and $ (F_2,\tau_{_2}) $ are two $ LGT $-spaces. A one-to-one map $ \phi $ from $ F_1 $ onto $ F_2 $ is called $ LGT $ isomorphism (briefly, isomorphism), if $ \phi $ and $ \phi^{-1} $ are $ LG $ maps, then we say $ (F_1, \tau_{_1}) $ and $ (F_2,\tau_{_2}) $ are isomorphic.
	\end{Def}
	
	\begin{Thm}
		Suppose that $ (F_1,\tau_{_1}) $ and $ (F_2,\tau_{_2}) $ are two $ LGT $-spaces and $ \phi $ is a join preserve one-to-one map from $ F_1 $ onto $ F_2 $. Then the following are equivalent
		\begin{itemize}
			\item[(a)] $ \phi $ is an isomorphism.
			\item[(b)] $ \phi $ is an open $ LG $ map.
			\item[(c)] $ \phi^{-1} $ is an open $ LG $ map.
		\end{itemize}
		\label{isomorphism}
	\end{Thm}
	\begin{proof}
		It follows immediately from Lemma \ref{one-to-one and preserving}, Corollary \ref{two sided and morphism} and these facts that $ \phi_* = \phi^{-1} $ and $ (\phi^{-1})_* = \phi $.
	\end{proof}

	\begin{Thm}
		Let  $ \{ (‌F_\alpha,\tau_\alpha) \}_{\alpha \in A} $ be a family of $ LGT $-spaces and  $ (‌ \prod_{\alpha \in A} F_\alpha, \tau_{_P} ) $ be the product $ LGT $-space. Then for each $ \beta \in A $, there are some subspaces of $ \prod_{\alpha \in A} F_\alpha $ which are  isomorphic to $ F_\beta $.
		\label{Component is isomorphim to a subspace}
	\end{Thm}
	\begin{proof}
		Set $ F =  \prod_{\alpha \in A} F_\alpha   $ and $ a \in F $ such that 
		\[   a_\alpha  = \begin{cases}
		0  & \alpha \neq \beta \\
		1  &  \alpha = \beta 
		\end{cases} \]
		Then $ \pi_{_\beta}|_{_{F_a}} : F_a \rightarrow F_\beta $ is $ OLG $, by  Theorems \ref{restriction of projection} and \ref{restriction of continuous}. Clearly, $ \pi_{_\beta}|_{_{F_a}} $ is an one-to-one onto map. Now suppose that $ t \wedge a \in \tau_a  $, in which $ t  \in \tau_{_P} $. Then $ \pi_{_\beta}|_{_{F_a}} (t \wedge 1) =  t_\beta \in \tau_{_\beta} $, so $ \pi_{_\beta}|_{_{F_a}} $ is open, and therefore $ \pi_{_\beta}|_{_{F_a}} $ is an isomorphism map, by Proposition \ref{two sided and morphism} and Theorem \ref{isomorphism}. Hence $ F_\beta $ is isomorphism to the subspace $ F_a $ of $  \prod_{\alpha \in A} F_\alpha $
	\end{proof}
	
	\begin{Def}
		A property is called $ LG $-property if it preserves by isomorphism.
	\end{Def}
	\begin{Thm}
		Compactness, countably compactness and Lindel\"of, $ ps $-property, $ T_0 $, $ T_1 $, $ T_2 $, regular and $ T_3 $ properties are $ LG $-properties.
		\label{LG property}
	\end{Thm}
	\begin{proof}
		By  Lemma \ref{one-to-one and preserving} and Proposition \ref{image of compact}, The compactness, countably compactness and Lindel\"of property are $ LG $-properties. It is easy to proof that the $ ps $-property, $ T_0 $, $ T_1 $ and $ T_2 $ properties are $ LG $-properties. Finally, by the use of Lemma \ref{one-to-one and complement}, one can show that the regular property and therefore $ T_3 $ property are $ LG $-properties.
	\end{proof}
	
	In \cite[Proposition 4.14]{aliabad2015lg}, it has been shown that if the product $ LGT $-space of the family of $ LGT $-spaces $ \{ ( F_\alpha, \tau_{_\alpha } ) \}_{\alpha \in A} $ is $ T_0 $ ($ T_1 $), then $ F_\alpha $ is $ T_0 $ ($ T_1 $), for every $ \alpha \in A $. Now, by Theorems \ref{Component is isomorphim to a subspace} and \ref{LG property} and \cite[Proposition 4.12]{aliabad2015lg}, we can say that they are evident.

\end{document}